\documentclass[3p, authoryear, square]{elsarticle}
\usepackage{amsthm,amssymb,amsmath}
\allowdisplaybreaks
\usepackage{textcomp}
\usepackage{algorithm}
\usepackage{algpseudocode}
\usepackage{booktabs,siunitx}
\usepackage{mathtools}
\makeatletter
\newenvironment{breakablealgorithm}
  {
   \begin{center}
     \refstepcounter{algorithm}
     \hrule height.8pt depth0pt \kern2pt
     \renewcommand{\caption}[2][\relax]{
       {\raggedright\textbf{\ALG@name~\thealgorithm} ##2\par}%
       \ifx\relax##1\relax 
         \addcontentsline{loa}{algorithm}{\protect\numberline{\thealgorithm}##2}%
       \else 
         \addcontentsline{loa}{algorithm}{\protect\numberline{\thealgorithm}##1}%
       \fi
       \kern2pt\hrule\kern2pt
     }
  }{
     \kern2pt\hrule\relax
   \end{center}
  }
\makeatother
\usepackage{chngcntr}
\counterwithin{figure}{section}
\counterwithin{table}{section}
\counterwithin{equation}{section}
\counterwithin{algorithm}{section}
\usepackage[labelfont=bf]{caption}
\usepackage{adjustbox}

\theoremstyle{definition}
\newtheorem{theorem}{Theorem}[section]
\newtheorem{lemma}{Lemma}[section]

\usepackage[usenames, dvipsnames]{color}
\usepackage{hyperref}
\hypersetup{
  colorlinks=true,
  pdfinfo={
    CreationDate={D:20180516090844},
    ModDate={D:20180516090844},
  },
}

\newcommand{\cond}[1]{\operatorname{cond}\left(#1\right)}
\newcommand{\fl}[1]{\operatorname{fl}\left(#1\right)}
\newcommand{\bigO}[1]{\mathcal{O}\left(#1\right)}
\newcommand{\mach}{\mathbf{u}}
\newcommand{\db}[1]{
  \ifthenelse{\equal{#1}{1}}
             {\partial b}
             {\partial^{#1} b}
}
\newcommand{\cdb}[1]{
  \ifthenelse{\equal{#1}{1}}
             {\widehat{\partial b}}
             {\widehat{\partial^{#1} b}}
}

\makeatletter
\def\ps@pprintTitle{%
  \let\@oddhead\@empty
  \let\@evenhead\@empty
  \def\@oddfoot{\footnotesize\itshape
    Published in Applied Mathematics and Computation (\cite{Hermes2019})
      \hfill April 5, 2019}%
  \let\@evenfoot\@oddfoot}
\makeatother

\begin{document}
\hypersetup{
  urlcolor=MidnightBlue,
  linkcolor=MidnightBlue,
  citecolor=ForestGreen,
}

\begin{frontmatter}

\title{Compensated de Casteljau algorithm in \(K\) times the working precision}
\author[djh]{Danny Hermes}\ead{dhermes@berkeley.edu}
\address[djh]{UC Berkeley, 970 Evans Hall \#3840, Berkeley, CA 94720-3840 USA}

\begin{abstract}
In computer aided geometric design a polynomial is usually represented in
Bernstein form. This paper presents a family of compensated algorithms to
accurately evaluate a polynomial in Bernstein form with floating point
coefficients. The principle is to apply error-free transformations to
improve the traditional de Casteljau algorithm. At each stage of computation,
round-off error is passed on to first order errors, then to second order
errors, and so on. After the computation has been ``filtered'' \((K - 1)\)
times via this process, the resulting output is as accurate as the de Casteljau
algorithm performed in \(K\) times the working precision. Forward error
analysis and numerical experiments illustrate the accuracy of this family
of algorithms.
\end{abstract}


\end{frontmatter}

\section{Introduction}

In computer aided geometric design, polynomials are usually expressed in
Bernstein form. Polynomials in this form are usually evaluated by the
de Casteljau algorithm. This algorithm has a round-off error bound
which grows only linearly with degree, even though the number of
arithmetic operations grows quadratically. The Bernstein basis is
optimally suited (\cite{Farouki1987, Delgado2015, Mainar2005})
for polynomial evaluation; it is
typically more accurate than the monomial basis, for example in
Figure~\ref{fig:horner-inferior} evaluation via Horner's method produces
a jagged curve for points near a triple root, but the de Casteljau algorithm
produces a smooth curve. Nevertheless the de Casteljau
algorithm returns results arbitrarily less accurate than the working
precision \(\mach\) when evaluating \(p(s)\) is ill-conditioned.
The relative accuracy of the computed
evaluation with the de Casteljau algorithm (\texttt{DeCasteljau}) satisfies
(\cite{Mainar1999}) the following a priori bound:
\begin{equation}\label{de-casteljau-error}
  \frac{\left|p(s) - \mathtt{DeCasteljau}(p, s)\right|}{\left|p(s)\right|} \leq
  \cond{p, s} \times \bigO{\mach}.
\end{equation}
In the right-hand side of this inequality, \(\mach\) is the computing
precision and the condition number \(\cond{p, s} \geq 1\) only depends
on \(s\) and the Bernstein coefficients of \(p\) --- its expression will
be given further.

\begin{figure}
  \includegraphics[width=0.9375\textwidth]{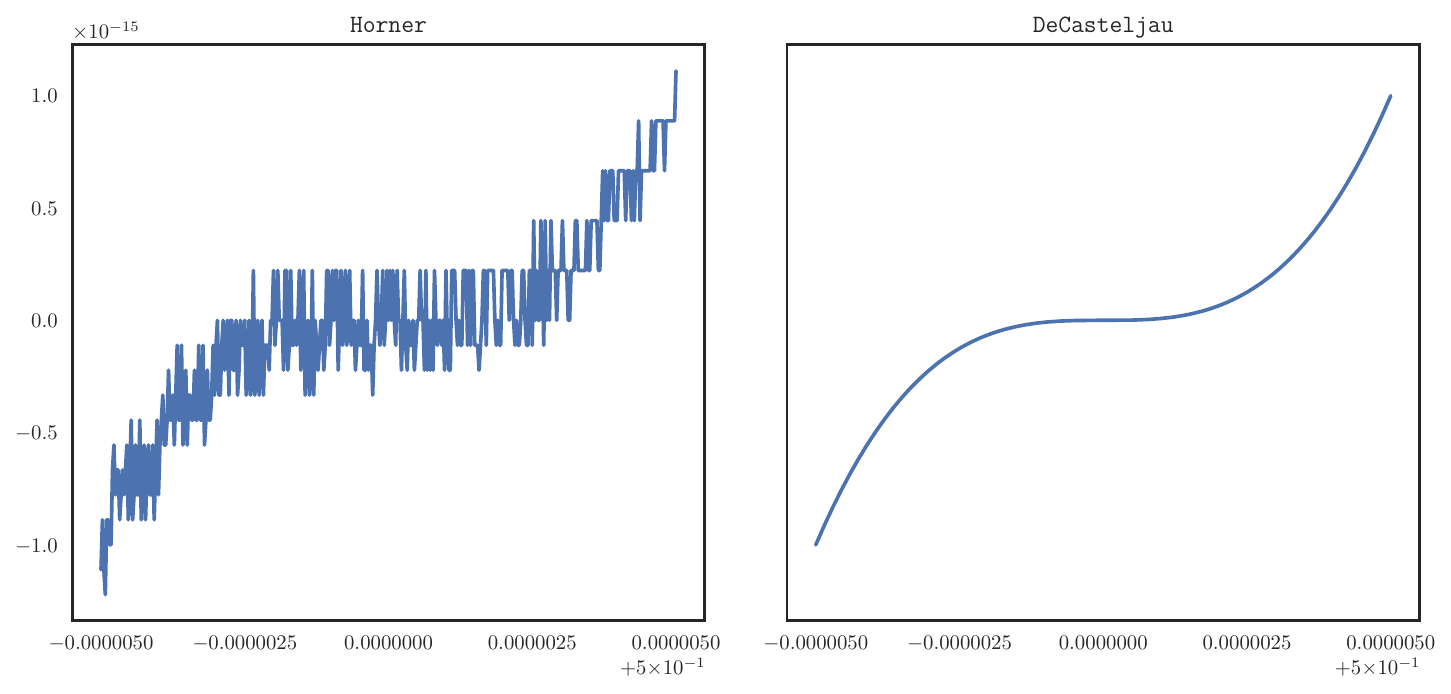}
  \centering
  \captionsetup{width=.75\linewidth}
  \caption{Comparing Horner's method to the de Casteljau method for
    evaluating \(p(s) = (2s - 1)^3\) in the neighborhood of its
    multiple root \(1/2\).}
  \label{fig:horner-inferior}
\end{figure}

For ill-conditioned problems, such as evaluating \(p(s)\) near a
multiple root, the condition number may be arbitrarily large, i.e.
\(\cond{p, s} > 1 / \mach\), in
which case most or all of the computed digits will be incorrect.
In some cases, even the order of magnitude of the computed value
of \(p(s)\) can be incorrect.

To address ill-conditioned problems, error-free transformations (EFT) can
be applied in \textit{compensated algorithms} to account for round-off.
Error-free transformations were studied in great detail in \cite{Ogita2005}
and open a large number of applications.
In \cite{langlois_et_al:DSP:2006:442}, a compensated Horner's algorithm was
described to evaluate a polynomial in the monomial basis. In \cite{Jiang2010},
a similar method was described to perform a compensated version of the de
Casteljau algorithm. In both cases, the \(\cond{p, s}\) factor is moved
from \(\mach\) to \(\mach^2\) and the computed value is as accurate
as if the computations were done in twice the working precision. For example,
the compensated de Casteljau algorithm (\texttt{CompDeCasteljau}) satisfies
\begin{equation}\label{de-casteljau-2-error}
  \frac{\left|p(s) - \mathtt{CompDeCasteljau}(p, s)\right|}{
    \left|p(s)\right|} \leq \mach + \cond{p, s} \times
    \bigO{\mach^2}.
\end{equation}
For problems with \(\cond{p, s} < 1 / \mach^2\), the relative error
is \(\mach\), i.e. accurate to full precision, aside from rounding to the
nearest floating point number. Figure~\ref{fig:jlcs-10} shows this shift
in relative error from \texttt{DeCasteljau} to \texttt{CompDeCasteljau}.

\begin{figure}
  \includegraphics[width=0.8125\textwidth]{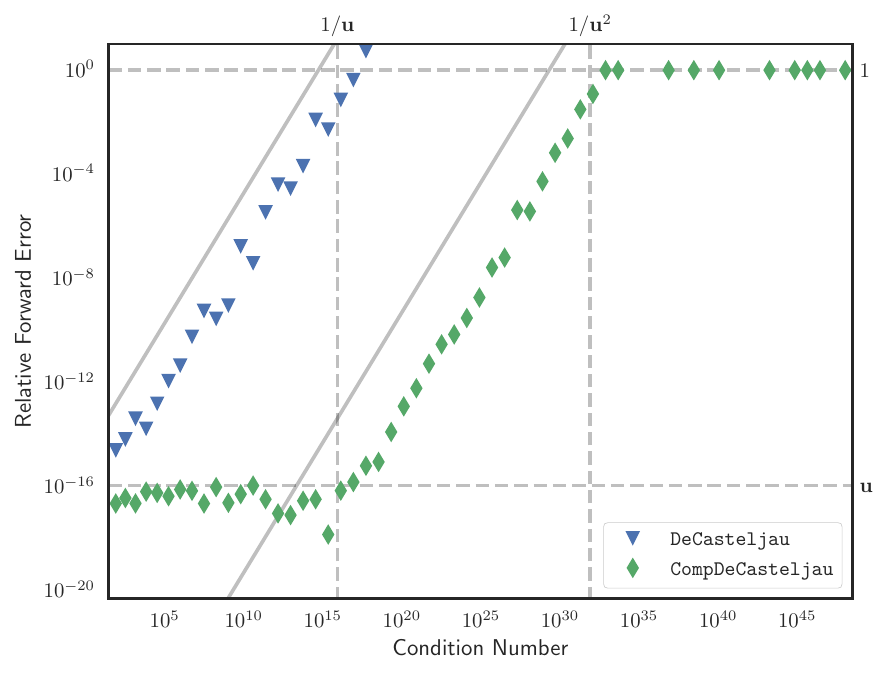}
  \centering
  \captionsetup{width=.75\linewidth}
  \caption{Evaluation of \(p(s) = (s - 1)\left(s - 3/4\right)^7\)
    represented in Bernstein form.}
  \label{fig:jlcs-10}
\end{figure}

In \cite{Graillat2009}, the authors generalized the compensated Horner's
algorithm to produce a method for evaluating a polynomial as if
the computations were done in \(K\) times the working precision for
any \(K \geq 2\). This result motivates this paper, though the
approach there is somewhat different than ours. They perform each computation
with error-free transformations and interpret the errors as coefficients of new
polynomials. They then evaluate the error polynomials, which (recursively)
generate second order error polynomials and so on. This recursive property
causes the number of operations to grow exponentially in \(K\). Here, we
instead have a fixed number of error groups, each corresponding to round-off
from the group above it. For example, when
\((1 - s) b_j^{(n)} + s b_{j + 1}^{(n)}\) is computed in floating point, any
error is filtered down to the error group below it.

As in~\eqref{de-casteljau-error}, the accuracy of the compensated
result~\eqref{de-casteljau-2-error} may be arbitrarily bad for ill-conditioned
polynomial evaluations. For example, as the condition number grows in
Figure~\ref{fig:jlcs-10}, some points have relative error exactly equal to
\(1\); this indicates that \(\mathtt{CompDeCasteljau}(p, s) = 0\), which is
a complete failure to evaluate the order of magnitude of \(p(s)\). For
root-finding problems \(\mathtt{CompDeCasteljau}(p, s) = 0\) when
\(p(s) \neq 0\) can cause premature convergence and incorrect results.
We describe how to defer rounding into progressively
smaller error groups and improve the accuracy of the computed result by a
factor of \(\mach\) for every error group added. So we derive
\texttt{CompDeCasteljauK}, a \(K\)-fold compensated de Casteljau algorithm
that satisfies the following a priori bound for any arbitrary integer \(K\):
\begin{equation}
  \frac{\left|p(s) - \mathtt{CompDeCasteljauK}(p, s, K)\right|}{
    \left|p(s)\right|} \leq \mach + \cond{p, s} \times
    \bigO{\mach^K}.
\end{equation}
This means that the computed value with \texttt{CompDeCasteljauK} is now
as accurate as the result of the de Casteljau algorithm performed in
\(K\) times the working precision with a final rounding back to the
working precision.

The paper is organized as follows. Section~\ref{sec:notation} establishes
notation for error analysis with floating point operations, reviews
results about error-free transformations and reviews the
de Casteljau algorithm. In Section~\ref{sec:compensated-2},
the compensated algorithm for polynomial evaluation from \cite{Jiang2010} is
reviewed and notation is established for the expansion. In
Section~\ref{sec:compensated-k}, the \(K\)-compensated algorithm is provided
and a forward error analysis is performed. Finally, in
Section~\ref{sec:numerical} we perform two numerical experiments to
give practical examples of the theoretical error bounds.

\section{Basic notation and results}\label{sec:notation}

\subsection{Floating Point and Forward Error Analysis}

We assume all floating point operations obey
\begin{equation}
  a \star b = \fl{a \circ b} = (a \circ b)(1 + \delta_1) =
  (a \circ b) / (1 + \delta_2)
\end{equation}
where \(\star \in \left\{\oplus, \ominus, \otimes, \oslash\right\}\), \(\circ
\in \left\{+, -, \times, \div\right\}\) and \(\left|\delta_1\right|,
\left|\delta_2\right| \leq \mach\). The symbol \(\mach\) is the unit round-off
and \(\star\) is a floating point operation, e.g.
\(a \oplus b = \fl{a + b}\). (For IEEE-754 floating point double precision,
\(\mach = 2^{-53}\).) We denote the computed result of
\(\alpha \in \mathbf{R}\) in floating point arithmetic by
\(\widehat{\alpha}\) or \(\fl{\alpha}\) and use \(\mathbf{F}\) as the set of
all floating point numbers (see \cite{Higham2002} for more details).
Following \cite{Higham2002}, we will use the following classic properties in
error analysis.

\begin{enumerate}
  \item If \(\delta_i \leq \mach\), \(\rho_i = \pm 1\), then
      \(\prod_{i = 1}^n (1 + \delta_i)^{\rho_i} = 1 + \theta_n\),
  \item \(\left|\theta_n\right| \leq \gamma_n \coloneqq
      n \mach / (1 - n \mach)\),
  \item \((1 + \theta_k)(1 + \theta_j) = 1 + \theta_{k + j}\),
  \item \(\gamma_k + \gamma_j + \gamma_k \gamma_j \leq \gamma_{k + j}
    \Longleftrightarrow (1 + \gamma_k)(1 + \gamma_j) \leq 1 + \gamma_{k + j}\),
  \item \((1 + \mach)^j \leq 1 / (1 - j \mach) \Longleftrightarrow
  (1 + \mach)^j - 1 \leq \gamma_j\).
\end{enumerate}

\subsection{Error-Free Transformation}

An error-free transformation is a computational method where both
the computed result and the round-off error are returned. It
is considered ``free'' of error if the round-off can be represented
exactly as an element or elements of \(\mathbf{F}\).
The error-free transformations used in this paper are
the \texttt{TwoSum} algorithm by Knuth (\cite{Knuth1969}) and
\texttt{TwoProd} algorithm by Dekker (\cite{Dekker1971}, Section 5),
respectively.

\begin{theorem}[\cite{Ogita2005}, Theorem 3.4]\label{thm:eft}
For \(a, b \in \mathbf{F}\) and \(P, \pi, S, \sigma \in \mathbf{F}\),
\texttt{TwoSum} and \texttt{TwoProd} satisfy
\begin{alignat}{4}
\left[S, \sigma\right] &= \mathtt{TwoSum}(a, b), & \, S &= \fl{a + b},
  S + \sigma &= a + b, \sigma &\leq \mach \left|S\right|,
  & \, \sigma &\leq \mach \left|a + b\right| \\
\left[P, \pi\right] &= \mathtt{TwoProd}(a, b),
  & \, P &= \fl{a \times b}, P + \pi &= a \times b,
  \pi &\leq \mach \left|P\right|,
  & \, \pi &\leq \mach \left|a \times b\right|.
\end{alignat}
The letters \(\sigma\) and \(\pi\) are used to indicate that the
errors came from sum and product, respectively. See
Appendix~\ref{sec:appendix-algo} for implementation details.
\end{theorem}

\subsection{de Casteljau Algorithm}

Next, we recall\footnote{We have used slightly non-standard notation for the
terms produced by the de Casteljau algorithm: we start the superscript at
\(n\) and count down to \(0\) as is typically done when describing Horner's
algorithm. For example, we use \(b_j^{(n - 2)}\) instead of
\(b_j^{(2)}\).} the de Casteljau algorithm:

\begin{breakablealgorithm}
  \caption{\textit{de Casteljau algorithm for polynomial evaluation.}}
  \label{alg:de-casteljau}

  \begin{algorithmic}
    \Function{\(\mathtt{result} = \mathtt{DeCasteljau}\)}{$b, s$}
      \State \(n = \texttt{length}(b) - 1\)
      \State \(\widehat{r} = 1 \ominus s\)
      \\
      \For{\(j = 0, \ldots, n\)}
        \State \(\widehat{b}_j^{(n)} = b_j\)
      \EndFor
      \\
      \For{\(k = n - 1, \ldots, 0\)}
        \For{\(j = 0, \ldots, k\)}
          \State \(\widehat{b}_j^{(k)} = \left(
              \widehat{r} \otimes \widehat{b}_j^{(k + 1)}\right) \oplus
              \left(s \otimes \widehat{b}_{j + 1}^{(k + 1)}\right)\)
        \EndFor
      \EndFor
      \\
      \State \(\mathtt{result} = \widehat{b}_0^{(0)}\)
    \EndFunction
  \end{algorithmic}
\end{breakablealgorithm}

\begin{theorem}[\cite{Mainar1999}, Corollary 3.2]
If \(p(s) = \sum_{j = 0}^n b_j B_{j, n}(s)\) and \(\mathtt{DeCasteljau}(p, s)\)
is the value computed by the de Casteljau algorithm then\footnote{In the
original paper the factor on \(\widetilde{p}(s)\) is \(\gamma_{2n}\),
but the authors did not consider round-off when computing
\(1 \ominus s\).}
\begin{equation}
\left|p(s) - \mathtt{DeCasteljau}(p, s)\right| \leq \gamma_{3n}
\sum_{j = 0}^n \left|b_j\right| B_{j, n}(s).
\end{equation}
\end{theorem}

The relative condition number of the evaluation of \(p(s) = \sum_{j = 0}^n
b_j B_{j, n}(s)\) in Bernstein form used in this paper is (see
\cite{Mainar1999}, \cite{Farouki1987}):
\begin{equation}
\cond{p, s} = \frac{\widetilde{p}(s)}{\left|p(s)\right|},
\end{equation}
where \(B_{j, n}(s) = \binom{n}{j} (1 - s)^{n - j} s^j \geq 0\) and
\(\widetilde{p}(s) \coloneqq \sum_{j = 0}^n \left|b_j\right| B_{j, n}(s)\).

To be able to express the algorithm in matrix form, we define
the vectors
\begin{equation}
b^{(k)} = \left[\begin{array}{c c c} b_0^{(k)} & \cdots &
b_k^{(k)}\end{array}\right]^T, \quad
\widehat{b}^{(k)} = \left[\begin{array}{c c c} \widehat{b}_0^{(k)} & \cdots &
    \widehat{b}_k^{(k)}\end{array}\right]^T
\end{equation}
and the reduction matrices:
\begin{equation}
U_k = U_k(s) = \left[\begin{array}{c c c c c c}
    1 - s  & s      & 0      & \cdots & \cdots & 0      \\
    0      & 1 - s  & s      & \ddots &        & \vdots \\
    \vdots & \ddots & \ddots & \ddots & \ddots & \vdots \\
    \vdots &        & \ddots & \ddots & \ddots & 0 \\
    0      & \cdots & \cdots & 0      & 1 - s  & s
\end{array}\right] \in \mathbf{R}^{k \times (k + 1)}.
\end{equation}
With this, we can express (\cite{Mainar1999}) the de Casteljau algorithm as
\begin{equation}\label{matrix-de-casteljau}
b^{(k)} = U_{k + 1} b^{(k + 1)}
\Longrightarrow b^{(0)} = U_1 \cdots U_n b^{(n)}.
\end{equation}

In general, for a sequence \(v_0, \ldots, v_n\) we'll refer to \(v\)
as the vector containing all of the values:
\(v = \left[\begin{array}{c c c} v_0 & \cdots &
    v_n\end{array}\right]^T.\)

\section{Compensated de Casteljau}\label{sec:compensated-2}

In this section we review the compensated de Casteljau algorithm
from \cite{Jiang2010}. In order to track the local errors at
each update step, we use four EFTs:
\begin{align}
\left[\widehat{r}, \rho\right] &= \mathtt{TwoSum}(1, -s) \\
\left[P_1, \pi_1\right] &= \mathtt{TwoProd}\left(
    \widehat{r}, \widehat{b}_j^{(k + 1)}\right) \\
\left[P_2, \pi_2\right] &= \mathtt{TwoProd}\left(
    s, \widehat{b}_{j + 1}^{(k + 1)}\right) \\
\left[\widehat{b}_j^{(k)}, \sigma_3\right] &= \mathtt{TwoSum}(P_1, P_2)
\end{align}
With these, we can exactly describe the local error between the exact
update and computed update:
\begin{gather}
\ell_{1, j}^{(k)} = \pi_1 + \pi_2 + \sigma_3 + \rho \cdot
  \widehat{b}_j^{(k + 1)} \label{ell-j} \\
(1 - s) \cdot \widehat{b}_j^{(k + 1)} +
  s \cdot \widehat{b}_{j + 1}^{(k + 1)} =
\widehat{b}_j^{(k)} + \ell_{1, j}^{(k)}.
\end{gather}

\begin{figure}
  \includegraphics[width=0.375\textwidth]{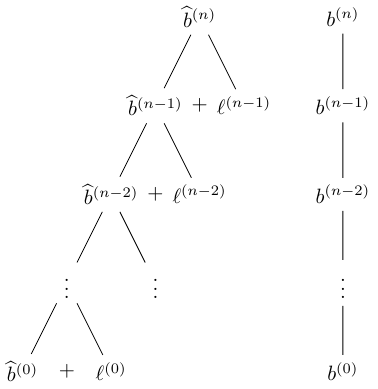}
  \centering
  \captionsetup{width=.75\linewidth}
  \caption{Local round-off errors}
  \label{fig:loc-err-accumulate}
\end{figure}

\noindent By defining the global errors at each step
\begin{equation}
  \db{1}_j^{(k)} = b_j^{(k)} - \widehat{b}_j^{(k)}
\end{equation}
we can see (Figure~\ref{fig:loc-err-accumulate}) that the local errors
accumulate in
\(\db{1}^{(k)}\):
\begin{equation}\label{err-update}
  \db{1}_j^{(k)} = (1 - s) \cdot \db{1}_j^{(k + 1)} + s \cdot
  \db{1}_{j + 1}^{(k + 1)} + \ell_{1, j}^{(k)}.
\end{equation}
When computed in exact arithmetic
\begin{equation}
  p(s) = \widehat{b}_0^{(0)} + \db{1}_0^{(0)}
\end{equation}
and by using \eqref{err-update}, we can continue to compute
approximations of \(\db{1}_j^{(k)}\). The idea behind the compensated
de Casteljau algorithm is to compute both the local error and the updates
of the global error with floating point operations:

\begin{breakablealgorithm}
  \caption{\textit{Compensated de Casteljau algorithm for polynomial evaluation.}}
  \label{alg:comp-de-casteljau}

  \begin{algorithmic}
    \Function{\(\mathtt{result} = \mathtt{CompDeCasteljau}\)}{$b, s$}
      \State \(n = \texttt{length}(b) - 1\)
      \State \(\left[\widehat{r}, \rho\right] = \mathtt{TwoSum}(1, -s)\)
      \\
      \For{\(j = 0, \ldots, n\)}
        \State \(\widehat{b}_j^{(n)} = b_j\)
        \State \(\cdb{1}_j^{(n)} = 0\)
      \EndFor
      \\
      \For{\(k = n - 1, \ldots, 0\)}
        \For{\(j = 0, \ldots, k\)}
          \State \(\left[P_1, \pi_1\right] = \mathtt{TwoProd}\left(
              \widehat{r}, \widehat{b}_j^{(k + 1)}\right)\)
          \State \(\left[P_2, \pi_2\right] = \mathtt{TwoProd}\left(
              s, \widehat{b}_{j + 1}^{(k + 1)}\right)\)
          \State \(\left[\widehat{b}_j^{(k)}, \sigma_3\right] =
              \mathtt{TwoSum}(P_1, P_2)\)
          \State \(\widehat{\ell}_{1, j}^{(k)} = \pi_1 \oplus \pi_2 \oplus
              \sigma_3 \oplus \left(\rho \otimes
              \widehat{b}_j^{(k + 1)}\right)\)
          \State \(\cdb{1}_j^{(k)} =
              \widehat{\ell}_{1, j}^{(k)} \oplus
              \left(s \otimes \cdb{1}_{j + 1}^{(k + 1)}
              \right) \oplus
              \left(\widehat{r} \otimes
              \cdb{1}_j^{(k + 1)}\right)\)
        \EndFor
      \EndFor
      \\
      \State \(\mathtt{result} = \widehat{b}_0^{(0)} \oplus
          \cdb{1}_0^{(0)}\)
    \EndFunction
  \end{algorithmic}
\end{breakablealgorithm}

\noindent  When comparing this computed error to the exact error, the
difference depends only on \(s\) and the Bernstein
coefficients of \(p\). Using a bound (Lemma~\ref{lemma:db-lemma}) on the
round-off error when computing \(\db{1}^{(0)}\), the algorithm can
be shown to be as accurate as if the computations were done in twice
the working precision:

\begin{theorem}[\cite{Jiang2010}, Theorem 5]
  If no underflow occurs, \(n \geq 2\) and \(s \in \left[0, 1\right]\)
  \begin{equation}
    \frac{\left|p(s) - \mathtt{CompDeCasteljau}(p, s)\right|}{
      \left|p(s)\right|} \leq \mach + 2 \gamma_{3n}^2 \cond{p, s}.
  \end{equation}
\end{theorem}

\begin{figure}
  \includegraphics[width=0.9375\textwidth]{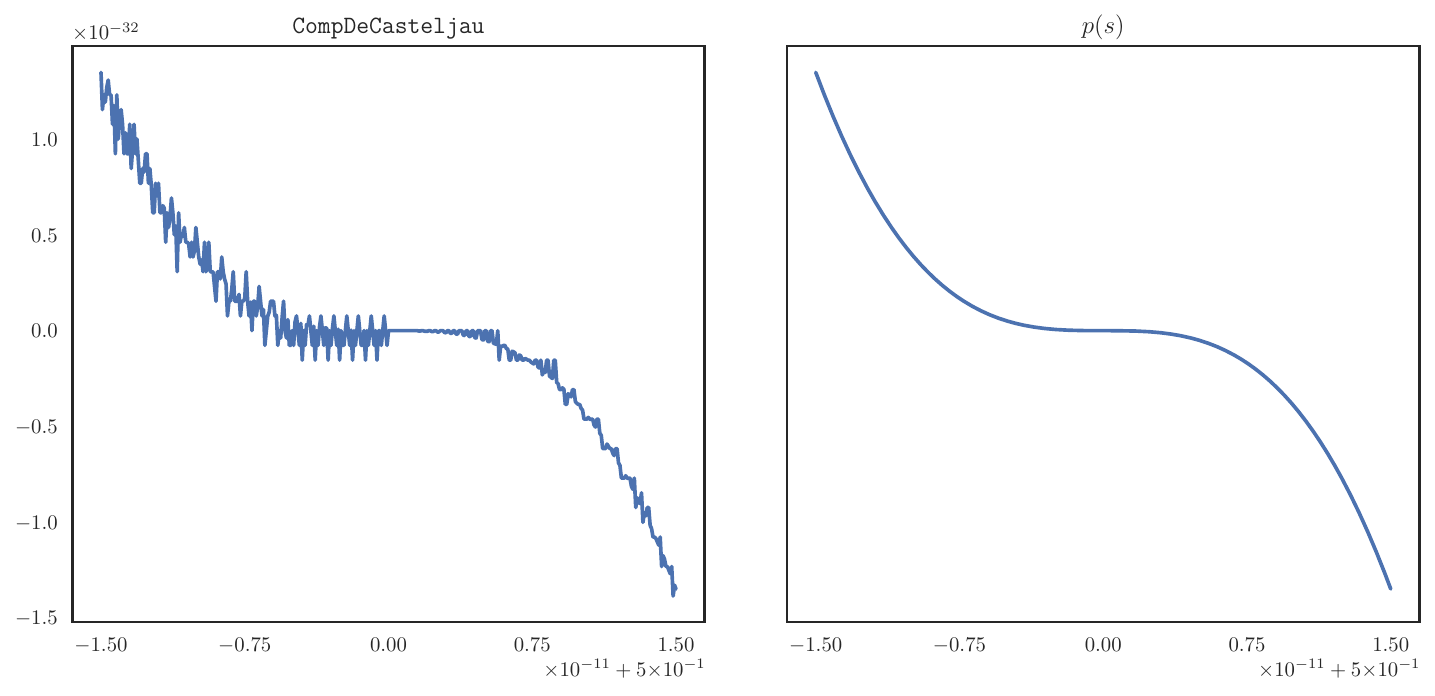}
  \centering
  \captionsetup{width=.75\linewidth}
  \caption{The compensated de Casteljau method starts to lose accuracy
    for \(p(s) = (2s - 1)^3 (s - 1)\) in the neighborhood of its
    multiple root \(1/2\).}
  \label{fig:compensated-insufficient}
\end{figure}

Unfortunately, Figure~\ref{fig:compensated-insufficient} shows how
\texttt{CompDeCasteljau} starts to break down in a region of
high condition number (caused by a multiple root with multiplicity
higher than two). For example, the point
\(s = \frac{1}{2} + 1001\mach\)
--- which is in the plotted region \(\left|s - \frac{1}{2}\right|
\leq \frac{3}{2} \cdot 10^{-11}\) --- evaluates to exactly \(0\) when
it should be \(\bigO{\mach^3}\). As shown in
Table~\ref{tab:exact-computation}, the breakdown occurs because
\(\widehat{b}_0^{(0)} = -\cdb{1}_0^{(0)} = \mach / 16\).

\begin{table}
  \centering
  \begin{adjustbox}{max width=\textwidth}
  \begin{tabular}{>{$}c<{$} >{$}c<{$} >{$}c<{$} >{$}c<{$} >{$}c<{$} >{$}c<{$}}
    \toprule
    k & j & \widehat{b}_j^{(k)} & \cdb{1}_j^{(k)} & \db{1}_j^{(k)} - \cdb{1}_j^{(k)} \\
    \midrule
    3 & 0 & 0.125 - 1.75 (1001 \mach) - 0.25 \mach & 0.25\mach & 0 \\
    3 & 1 & -0.125 + 1.25(1001 \mach) + 0.25 \mach & -0.25\mach & 0 \\
    3 & 2 & 0.125 - 0.75 (1001 \mach) & 0 & 0 \\
    3 & 3 & -0.125 + 0.25 (1001 \mach) & 0 & 0 \\
    \midrule
    2 & 0 & -0.5 (1001 \mach) & 3 (1001 \mach)^2 & 0 \\
    2 & 1 & 0.5(1001 \mach) + 0.125 \mach & -0.125\mach - 2 (1001 \mach)^2 & 0 \\
    2 & 2 & -0.5 (1001 \mach) & (1001 \mach)^2 & 0 \\
    \midrule
    1 & 0 & 0.0625\mach + (1001 \mach)^2 + 239\mach^2 & -0.0625\mach + 0.5  (1001 \mach)^2 - 239 \mach^2 & -5 (1001\mach)^3 \\
    1 & 1 & 0.0625\mach - (1001 \mach)^2 - 239\mach^2 & -0.0625\mach - 0.5  (1001 \mach)^2 + 239 \mach^2 & 3 (1001\mach)^3 \\
    \midrule
    0 & 0 & 0.0625 \mach & -0.0625 \mach & -4 (1001 \mach)^3 + 8 (1001 \mach)^4 \\
    \bottomrule
  \end{tabular}
  \end{adjustbox}
  \captionsetup{width=.75\linewidth}
  \caption{Terms computed by \texttt{CompDeCasteljau} when evaluating \\
    \(p(s) = (2s - 1)^3 (s - 1)\) at the point
    \(s = \frac{1}{2} + 1001 \mach\)}
  \label{tab:exact-computation}
\end{table}

\section{\texorpdfstring{\(K\)}{K}-Compensated de Casteljau}\label{sec:compensated-k}

\subsection{Algorithm Specified}

In order to raise from twice the working precision to \(K\) times the
working precision, we continue using EFTs when computing
\(\cdb{1}^{(k)}\). By tracking the round-off from each
floating point evaluation via an EFT, we can form a cascade of global errors:
\begin{align}
  b_j^{(k)} &= \widehat{b}_j^{(k)} + \db{1}_j^{(k)} \\
  \db{1}_j^{(k)} &= \cdb{1}_j^{(k)} + \db{2}_j^{(k)} \\
  \db{2}_j^{(k)} &= \cdb{2}_j^{(k)} +
  \db{3}_j^{(k)} \\
  &\mathrel{\makebox[\widthof{=}]{\vdots}} \nonumber
\end{align}
In the same way local error can be tracked when updating
\(\widehat{b}_j^{(k)}\), it can be tracked for updates that happen down
the cascade:
\begin{alignat}{4}
  (1 - s) \cdot \widehat{b}_j^{(k + 1)} &+
  s \cdot \widehat{b}_{j + 1}^{(k + 1)} &&  &&=
  \widehat{b}_j^{(k)} &&+ \ell_{1, j}^{(k)} \\
  (1 - s) \cdot \cdb{1}_j^{(k + 1)} &+
  s \cdot \cdb{1}_{j + 1}^{(k + 1)} &&+ \ell_{1, j}^{(k)} &&=
  \cdb{1}_j^{(k)} &&+ \ell_{2, j}^{(k)} \\
  (1 - s) \cdot \cdb{2}_j^{(k + 1)} &+
  s \cdot \cdb{2}_{j + 1}^{(k + 1)} &&+ \ell_{2, j}^{(k)} &&=
  \cdb{2}_j^{(k)} &&+ \ell_{3, j}^{(k)} \\
  &  &&  &&\mathrel{\makebox[\widthof{=}]{\vdots}} && \nonumber
\end{alignat}

\begin{figure}
  \includegraphics[width=0.875\textwidth]{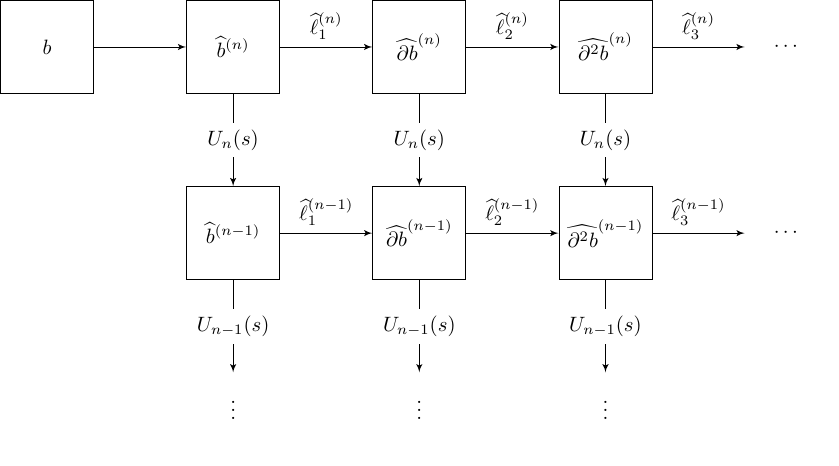}
  \centering
  \captionsetup{width=.75\linewidth}
  \caption{Filtering errors}
  \label{fig:error-filtration}
\end{figure}

In \texttt{CompDeCasteljau} (Algorithm~\ref{alg:comp-de-casteljau}), after
a single stage of error filtering we
``give up'' and use \(\cdb{1}\) instead of
\(\db{1}\) (without keeping around any information about the
round-off error). In order to obtain results that are as accurate as if
computed in \(K\) times the working precision, we must continue filtering
(see Figure~\ref{fig:error-filtration})
errors down \((K - 1)\) times, and only at the final level do we accept
the rounded
\(\cdb{K - 1}\) in place of the exact
\(\db{K - 1}\).

When computing \(\cdb{F}\) (i.e. the error after
\(F\) stages of filtering)
there will be several sources of round-off. In particular, there will be
\begin{itemize}
\item errors when computing \(\widehat{\ell}_{F, j}^{(k)}\) from the
  terms in \(\ell_{F, j}^{(k)}\)
\item an error
for the ``missing'' \(\rho \cdot \cdb{F}_j^{(k + 1)}\) in
\((1 - s) \cdot \cdb{F}_j^{(k + 1)}\)
\item an error from the product
  \(\widehat{r} \otimes \cdb{F}_j^{(k + 1)}\)
\item an error from the product
  \(s \otimes \cdb{F}_{j + 1}^{(k + 1)}\)
\item two errors from the two \(\oplus\) when combining the three
  terms in
  \(\widehat{\ell}_{F, j}^{(k)} \oplus
  \left(s \otimes \cdb{F}_{j + 1}^{(k + 1)}\right) \oplus
  \left(\widehat{r} \otimes \cdb{F}_j^{(k + 1)}\right)\)
\end{itemize}
For example, in~\eqref{ell-j}:
\begin{equation}
\ell_{1, j}^{(k)} =
    \underbrace{\vphantom{\rho \cdot \widehat{b}_j^{(k + 1)}} \pi_1}_{
        \vphantom{(1 - s) \widehat{b}_j^{(k + 1)}}
        P_1 = \widehat{r} \otimes \widehat{b}_j^{(k + 1)}} +
    \underbrace{\vphantom{\rho \cdot \widehat{b}_j^{(k + 1)}} \pi_2}_{
        \vphantom{(1 - s) \widehat{b}_j^{(k + 1)}}
        P_2 = s \otimes \widehat{b}_{j + 1}^{(k + 1)}} +
    \underbrace{\vphantom{\rho \cdot \widehat{b}_j^{(k + 1)}} \sigma_3}_{
        \vphantom{(1 - s) \widehat{b}_j^{(k + 1)}}
        P_1 \oplus P_2} +
    \underbrace{\rho \cdot \widehat{b}_j^{(k + 1)}}_{
        (1 - s) \widehat{b}_j^{(k + 1)}}
\end{equation}
After each stage, we'll always have
\begin{equation}
\ell_{F, j}^{(k)} = e_1 + \cdots + e_{5F - 2} + \rho \cdot
\cdb{F - 1}_j^{(k + 1)}
\end{equation}
where the terms \(e_1, \ldots, e_{5F - 2}\) come from using \texttt{TwoSum}
and \texttt{TwoProd} when computing \(\cdb{F - 1}_j^{(k)}\)
and the \(\rho\) term comes from the round-off
in \(1 \ominus s\) when multiplying \((1 - s)\) by
\(\cdb{F - 1}_j^{(k + 1)}\). With this in mind, we
can define an EFT (\texttt{LocalErrorEFT}) that computes
\(\widehat{\ell}\) and tracks all round-off errors generated in
the process:

\begin{breakablealgorithm}
  \caption{\textit{EFT for computing the local error.}}
  \label{alg:local-error-eft}

  \begin{algorithmic}
    \Function{\(\left[\eta, \widehat{\ell}\right] =
        \mathtt{LocalErrorEFT}\)}{$e, \rho, \delta b$}
      \State \(L = \texttt{length}(e)\)
      \\
      \State \(\left[\widehat{\ell}, \eta_1\right] =
          \mathtt{TwoSum}(e_1, e_2)\)
      \For{\(j = 3, \ldots, L\)}
        \State \(\left[\widehat{\ell}, \eta_{j - 1}\right] =
            \mathtt{TwoSum}\left(\widehat{\ell}, e_j\right)\)
      \EndFor
      \\
      \State \(\left[P, \eta_L\right] =
          \mathtt{TwoProd}\left(\rho, \delta b\right)\)
      \State \(\left[\widehat{\ell}, \eta_{L + 1}\right] =
          \mathtt{TwoSum}\left(\widehat{\ell}, P\right)\)
    \EndFunction
  \end{algorithmic}
\end{breakablealgorithm}

\noindent With this EFT in place\footnote{And the related
\texttt{LocalError} in Algorithm~\ref{alg:local-error}}, we can
perform \((K - 1)\) error filtrations. Once we've computed the \(K\) stages
of global errors, they can be combined with
\texttt{SumK} (Algorithm~\ref{alg:sum-k}) to produce a sum that is as
accurate as if computed in \(K\) times the working precision.

\begin{breakablealgorithm}
  \caption{\(K\)-\textit{compensated de Casteljau algorithm.}}
  \label{alg:k-comp-de-casteljau}

  \begin{algorithmic}
    \Function{\(\mathtt{result} = \mathtt{CompDeCasteljauK}\)}{$b, s, K$}
      \State \(n = \texttt{length}(b) - 1\)
      \State \(\left[\widehat{r}, \rho\right] = \mathtt{TwoSum}(1, -s)\)
      \\
      \For{\(j = 0, \ldots, n\)}
        \State \(\widehat{b}_j^{(n)} = b_j\)
        \For{\(F = 1, \ldots, K - 1\)}
          \State \(\cdb{F}_j^{(n)} = 0\)
        \EndFor
      \EndFor
      \\
      \For{\(k = n - 1, \ldots, 0\)}
        \For{\(j = 0, \ldots, k\)}
          \State \(\left[P_1, \pi_1\right] = \mathtt{TwoProd}\left(
              \widehat{r}, \widehat{b}_j^{(k + 1)}\right)\)
          \State \(\left[P_2, \pi_2\right] = \mathtt{TwoProd}\left(
              s, \widehat{b}_{j + 1}^{(k + 1)}\right)\)
          \State \(\left[\widehat{b}_j^{(k)}, \sigma_3\right] =
              \mathtt{TwoSum}(P_1, P_2)\)
          \\
          \State \(e = \left[\pi_1, \pi_2, \sigma_3\right]\)
          \State \(\delta b = \widehat{b}_j^{(k + 1)}\)
          \\
          \For{\(F = 1, \ldots, K - 2\)}
            \State \(\left[\eta, \widehat{\ell}\right] =
                \mathtt{LocalErrorEFT}(e, \rho, \delta b)\)
            \State \(L = \texttt{length}(\eta)\)
            \\
            \State \(\left[P_1, \eta_{L + 1}\right] = \mathtt{TwoProd}\left(
                s, \cdb{F}_{j + 1}^{(k + 1)}\right)\)
            \State \(\left[S_2, \eta_{L + 2}\right] =
                \mathtt{TwoSum}\left(\widehat{\ell}, P_1\right)\)
            \State \(\left[P_3, \eta_{L + 3}\right] = \mathtt{TwoProd}\left(
                \widehat{r}, \cdb{F}_j^{(k + 1)}\right)\)
            \State \(\left[\cdb{F}_j^{(k)}, \eta_{L + 4}\right]
                = \mathtt{TwoSum}\left(S_2, P_3\right)\)
            \\
            \State \(e = \eta\)
            \State \(\delta b = \cdb{F}_j^{(k + 1)}\)
          \EndFor
          \\
          \State \(\widehat{\ell} =
                \mathtt{LocalError}(e, \rho, \delta b)\)
          \State \(\cdb{K - 1}_j^{(k)} =
              \widehat{\ell} \oplus
              \left(s \otimes \cdb{K - 1}_{j + 1}^{(k + 1)}
              \right) \oplus
              \left(\widehat{r} \otimes
              \cdb{K - 1}_j^{(k + 1)}\right)\)
        \EndFor
      \EndFor
      \\
      \State \(\mathtt{result} = \mathtt{SumK}\left(\left[
        \widehat{b}_0^{(0)}, \ldots, \cdb{K - 1}_0^{(0)}\right], K\right)\)
    \EndFunction
  \end{algorithmic}
\end{breakablealgorithm}

\noindent Noting that \(\ell_{F, j}\) contains \(5F - 1\) terms, one can
show that \texttt{CompDeCasteljauK} (Algorithm~\ref{alg:k-comp-de-casteljau})
requires
\begin{equation}
(15K^2 + 11K - 34)T_n + 6K^2 - 11K + 11 =
\bigO{n^2 K^2}
\end{equation}
flops to evaluate a degree \(n\) polynomial, where \(T_n\) is the
\(n\)th triangular number. As a comparison, the non-compensated form of
de Casteljau requires \(3 T_n + 1\) flops. In total this will require
\((3K - 4)T_n\) uses of \texttt{TwoProd}. On hardware that supports
FMA, \texttt{TwoProdFMA} (Algorithm~\ref{alg:two-prod-fma}) can be used
instead, lowering the flop count by \(15(3K - 4)T_n\). Another way
to lower the total flop count is to just use
\(\widehat{b}_0^{(0)} \oplus \cdots \oplus \cdb{K - 1}_0^{(0)}\)
instead of \texttt{SumK}; this will reduce the total by
\(6(K - 1)^2\) flops. When using a standard sum, the results produced
are (empirically) identical to those with \texttt{SumK}. This makes
sense: the whole point of \texttt{SumK}
is to filter errors in a summation so that the final operation produces
a sum of the form \(v_1 \oplus \cdots \oplus v_K\) where each
term is smaller than the previous by a factor of \(\mach\). This
property is already satisfied for the \(\cdb{F}_0^{(0)}\) so in
practice the \(K\)-compensated summation is likely not needed.

\subsection{Error bound for polynomial evaluation}

\begin{theorem}[\cite{Ogita2005}, Proposition 4.10]\label{thm:sum-k}
A summation can be computed (\texttt{SumK}, Algorithm~\ref{alg:sum-k})
with results that are as accurate as if computed in \(K\) times the
working precision. When computed this way, the result satisfies:
\begin{equation}
\left|\mathtt{SumK}(v, K) - \sum_{j = 1}^n v_j\right| \leq
\left(\mach + 3 \gamma_{n - 1}^2\right) \left|\sum_{j = 1}^n v_j\right| +
\gamma_{2n - 2}^K \sum_{j = 1}^n \left|v_j\right|.
\end{equation}
\end{theorem}

\begin{lemma}[\cite{Jiang2010}, Theorem 4]\label{lemma:db-lemma}
The second order error \(\db{2}^{(0)}_0\) satisfies\footnote{The authors
  missed one round-off error so used \(\gamma_{3n + 1}\) where
  \(\gamma_{3n + 2}\) would have followed from their arguments.}
\begin{equation}
  \left|\db{1}^{(0)}_0 - \cdb{1}^{(0)}_0\right| =
  \left|\db{2}^{(0)}_0\right| \leq 2 \gamma_{3n + 2} \gamma_{3(n - 1)}
  \widetilde{p}(s).
\end{equation}
\end{lemma}

To enable a bound on the \(K\) order error \(\db{K}^{(0)}_0\), it's necessary
to understand the difference between the exact local errors \(\ell_{F, j}\)
and the computed equivalents \(\widehat{\ell}_{F, j}\). To do this, we define
\begin{equation}
\widetilde{\ell}_{F, j} \coloneqq \left|e_1\right| +
\cdots + \left|e_{5F - 2}\right| + \left|\rho \cdot
\cdb{F - 1}_j^{(k + 1)}\right|.
\end{equation}

\begin{lemma}\label{lemma:ell-tilde}
The local error bounds \(\widetilde{\ell}_{F, j}\) satisfy:
\begin{align}
\widetilde{\ell}_{1, j}^{(k)} &\leq
  \gamma_3 \left(
  (1 - s) \left|\widehat{b}_j^{(k + 1)}\right| +
  s \left|\widehat{b}_{j + 1}^{(k + 1)}\right|\right)
  \label{ell-tilde-1} \\
\widetilde{\ell}_{F + 1, j}^{(k)} &\leq
  \gamma_3 \left(
  (1 - s) \left|\cdb{F}_j^{(k + 1)}\right| +
  s \left|\cdb{F}_{j + 1}^{(k + 1)}\right|\right) +
  \gamma_{5F} \cdot \widetilde{\ell}_{F, j}^{(k)}
  \text{ for } F \geq 1.
\end{align}
\end{lemma}

As we'll see soon (Lemma~\ref{lemma:k-order}), putting a bound on
sums of the form \(\sum_{j = 0}^k \widetilde{\ell}_{F, j}^{(k)} B_{j, k}(s)\) will
be useful to get an overall bound on the relative error for
\texttt{CompDeCasteljauK}, so we define
\(L_{F, k} \coloneqq \sum_{j = 0}^k \widetilde{\ell}_{F, j}^{(k)} B_{j, k}(s)\).

\begin{lemma}\label{lemma:L-and-D-bounds}
For \(s \in \left[0, 1\right]\), the Bernstein-type error sum defined above
satisfies the following bounds:
\begin{align}
L_{F, n - k} &\leq \left[\left(3^F \binom{k}{F - 1} + \bigO{k^{F - 1}}\right)
  \mach^F + \bigO{\mach^{F + 1}}\right] \cdot \widetilde{p}(s) \\
\sum_{k = 0}^{n - 1} \gamma_{3k + 5F} L_{F, k} &\leq
  \left[\left(3^{F + 1} \binom{n}{F + 1} + \bigO{n^F}\right)
  \mach^{F + 1} + \bigO{\mach^{F + 2}}\right] \cdot \widetilde{p}(s).
  \label{L-sum-bound}
\end{align}
In particular, this means that
\(\sum_{k = 0}^{n - 1} \gamma_{3k + 5F} L_{F, k} =
\bigO{(3 n \mach)^{F + 1}} \cdot \widetilde{p}(s)\).
\end{lemma}

See Appendix~\ref{sec:appendix-proof-details} for details on
proving Lemma~\ref{lemma:ell-tilde} and Lemma~\ref{lemma:L-and-D-bounds}.

\begin{lemma}\label{lemma:k-order}
The \(K\) order error \(\db{K}^{(0)}_0\) satisfies
\begin{equation}
  \left|\db{K - 1}^{(0)}_0 - \cdb{K - 1}^{(0)}_0\right| =
  \left|\db{K}^{(0)}_0\right| \leq
  \left[\left(3^{K} \binom{n}{K} + \bigO{n^{K - 1}}\right)
  \mach^{K} + \bigO{\mach^{K + 1}}\right] \cdot \widetilde{p}(s).
\end{equation}
\end{lemma}

\begin{proof}
As in \eqref{matrix-de-casteljau}, we can express the compensated
de Casteljau algorithm as
\begin{equation}
\db{F}^{(k)} = U_{k + 1} \db{F}^{(k + 1)} + \ell_{F}^{(k)}
\Longrightarrow \db{F}^{(0)} = \sum_{k = 0}^{n - 1}
U_1 \cdots U_k \ell_F^{(k)} = \sum_{k = 0}^{n - 1}
\left[\sum_{j = 0}^k \ell_{F, j}^{(k)} B_{j, k}(s)\right].
\end{equation}
For the inexact equivalent of these things, first note that
\(\widehat{r} = (1 - s)(1 + \delta)\). Due to this,
we put the \(\widehat{r}\) term at the end of each update step to reduce
the amount of round-off:
\begin{align}
  \cdb{F}_j^{(k)} &=
  \widehat{\ell}_{F, j}^{(k)} \oplus
  \left(s \otimes \cdb{F}_{j + 1}^{(k + 1)}\right) \oplus
  \left(\widehat{r} \otimes \cdb{F}_j^{(k + 1)}\right) \\
&= (1 - s) \cdot \cdb{F}_j^{(k + 1)}(1 + \theta_3) +
  s \cdot \cdb{F}_{j + 1}^{(k + 1)}(1 + \theta_3) +
  \widehat{\ell}_{F, j}^{(k)} (1 + \theta_2) \\
\Longrightarrow \cdb{F}^{(k)} &=
  U_{k + 1} \cdb{F}^{(k + 1)}(1 + \theta_3) +
  \widehat{\ell}_{F}^{(k)} (1 + \theta_2) \\
\Longrightarrow \cdb{F}^{(0)} &=
  \sum_{k = 0}^{n - 1}
  U_1 \cdots U_k \widehat{\ell}_F^{(k)} (1 + \theta_{3k + 2})
  = \sum_{k = 0}^{n - 1}
  \left[\sum_{j = 0}^k \widehat{\ell}_{F, j}^{(k)} (1 + \theta_{3k + 2})
    B_{j, k}(s)\right].
\end{align}
Since
\begin{equation}
\db{F + 1}_0^{(0)} = \db{F}_0^{(0)} - \cdb{F}_0^{(0)} = \sum_{k = 0}^{n - 1}
\sum_{j = 0}^k \left(\ell_{F, j}^{(k)} -
\widehat{\ell}_{F, j}^{(k)} (1 + \theta_{3k + 2})\right) B_{j, k}(s)
\end{equation}
it's useful to put a bound on \(\ell_{F, j}^{(k)} -
\widehat{\ell}_{F, j}^{(k)} (1 + \theta_{3k + 2})\). Via
\begin{align}
\widehat{\ell}_{F, j}^{(k)} &= e_1 \oplus \cdots \oplus e_{5F - 2} \oplus
\left(\rho \otimes \cdb{F - 1}_j^{(k + 1)}\right) \\
&= e_1\left(1 + \theta_{5F - 2}\right) + \cdots +
e_{5F - 2}\left(1 + \theta_2\right) +
\rho \cdot \cdb{F - 1}_j^{(k + 1)} \left(1 + \theta_2\right)
\end{align}
we see that
\begin{equation}
\left|\ell_{F, j}^{(k)} -
\widehat{\ell}_{F, j}^{(k)} (1 + \theta_{3k + 2})\right| \leq
\gamma_{3k + 5F} \cdot \widetilde{\ell}_{F, j}^{(k)}
\Longrightarrow
\left|\db{F + 1}_0^{(0)}\right| \leq \sum_{k = 0}^{n - 1}
\gamma_{3k + 5F} \sum_{j = 0}^k \widetilde{\ell}_{F, j}^{(k)} B_{j, k}(s).
\end{equation}
Applying \eqref{L-sum-bound} directly gives
\begin{equation}
\left|\db{F + 1}_0^{(0)}\right| \leq
  \left[\left(3^{F + 1} \binom{n}{F + 1} + \bigO{n^F}\right)
  \mach^{F + 1} + \bigO{\mach^{F + 2}}\right] \cdot \widetilde{p}(s).
\end{equation}
Letting \(K = F + 1\) we have our result.
\end{proof}

\begin{theorem}
If no underflow occurs, \(n \geq 2\) and \(s \in \left[0, 1\right]\)
\begin{multline}
  \frac{\left|p(s) - \mathtt{CompDeCasteljau}(p, s, K)\right|}{
    \left|p(s)\right|} \leq \left[\mach + \bigO{\mach^2}
    \right] + \\
    \left[\left(3^{K} \binom{n}{K} + \bigO{n^{K - 1}}\right) \mach^K +
    \bigO{\mach^{K + 1}}\right] \cond{p, s}.
\end{multline}
\end{theorem}

\begin{proof}
Since
\begin{equation}
\mathtt{CompDeCasteljau}(p, s, K) = \mathtt{SumK}\left(\left[
  \widehat{b}_0^{(0)}, \ldots, \cdb{K - 1}_0^{(0)}\right], K\right),
\end{equation}
applying Theorem~\ref{thm:sum-k} tells us that
\begin{equation}\label{sum-k-applied}
\left|\mathtt{CompDeCasteljau}(p, s, K) - \sum_{F = 0}^{K - 1}
\cdb{F}_0^{(0)}\right| \leq
\left(\mach + 3 \gamma_{n - 1}^2\right) \left|\sum_{F = 0}^{K - 1}
\cdb{F}_0^{(0)}\right| +
\gamma_{2n - 2}^K \sum_{F = 0}^{K - 1} \left|\cdb{F}_0^{(0)}\right|.
\end{equation}
Since
\begin{equation}
p(s) = b_0^{(0)} = \widehat{b}_0^{(0)} + \db{1}_0^{(0)}
= \cdots
= \widehat{b}_0^{(0)} + \cdb{1}_0^{(0)} + \cdots
+ \cdb{K - 1}_0^{(0)} + \db{K}_0^{(0)}
\end{equation}
we have
\begin{gather}
\left|\sum_{F = 0}^{K - 1} \cdb{F}_0^{(0)}\right|
\leq \left|p(s)\right| + \left|\db{K}_0^{(0)}\right| \quad \text{and} \\
\left|\mathtt{CompDeCasteljau}(p, s, K) - p(s)\right| \leq
\left|\mathtt{CompDeCasteljau}(p, s, K) - \sum_{F = 0}^{K - 1}
\cdb{F}_0^{(0)}\right| +
\left|\db{K}_0^{(0)}\right| \label{triangle-ps}.
\end{gather}
Due to Lemma~\ref{lemma:k-order}, \(\db{F}_0^{(0)} =
\bigO{\mach^F} \widetilde{p}(s)\), hence
\begin{align}
\left(\mach + 3 \gamma_{n - 1}^2\right) \left|\sum_{F = 0}^{K - 1}
\cdb{F}_0^{(0)}\right| &\leq
\left[\mach + \bigO{\mach^2}\right] \left|p(s)\right| +
\bigO{\mach^{K + 1}} \widetilde{p}(s) \\
\gamma_{2n - 2}^K \sum_{F = 0}^{K - 1} \left|\cdb{F}_0^{(0)}\right| &\leq
\gamma_{2n - 2}^K \left|\widehat{b}_0^{(0)}\right| +
\bigO{\mach^{K + 1}} \widetilde{p}(s) \\
&\leq
\gamma_{2n - 2}^K \left[\left|p(s)\right| +
  \bigO{\mach} \widetilde{p}(s)\right] +
\bigO{\mach^{K + 1}} \widetilde{p}(s).
\end{align}
Combining this with \eqref{sum-k-applied} and \eqref{triangle-ps}, we
see
\begin{align}
& \left|\mathtt{CompDeCasteljau}(p, s, K) - p(s)\right| \\
\leq &
\left[\mach + \bigO{\mach^2}\right] \left|p(s)\right| +
\left|\db{K}_0^{(0)}\right| +
\bigO{\mach^{K + 1}} \widetilde{p}(s) \\
\leq &
\left[\mach + \bigO{\mach^2}\right] \left|p(s)\right| +
\left[\left(3^{K} \binom{n}{K} + \bigO{n^{K - 1}}\right) \mach^K +
\bigO{\mach^{K + 1}} \right]
\widetilde{p}(s).
\end{align}
Dividing this by \(\left|p(s)\right|\), we have our result.
\end{proof}

For the first few values of \(K\) the coefficient of
\(\cond{p, s}\) in the bound is
\begin{center}
  \begin{tabular}{>{$}c<{$} c >{$}c<{$}}
    K & Method & \text{Multiplier} \\
    \midrule
    1 & \texttt{DeCasteljau} & 3 \binom{n}{1} \mach =
      3n \mach \approx \gamma_{3n} \\[0.125cm]
    2 & \texttt{CompDeCasteljau} & \left[9 \binom{n}{2} +
      15 \binom{n}{1}\right]\mach^2 = \frac{3n(3n + 7)}{2} \mach^2
      \approx \frac{1}{4} \cdot 2 \gamma_{3n}^2 \\[0.125cm]
    3 & \texttt{CompDeCasteljau3} & \left[27 \binom{n}{3} +
      135 \binom{n}{2} + 150 \binom{n}{1}\right] \mach^3 =
      \frac{3n(3n^2 + 36n + 61)}{2} \mach^3 \\[0.125cm]
    4 & \texttt{CompDeCasteljau4} & \left[81 \binom{n}{4} + 810 \binom{n}{3} +
      2475 \binom{n}{2} + 2250 \binom{n}{1}\right] \mach^4 \\[0.125cm]
  \end{tabular}
\end{center}
See the \hyperref[proof:L-and-D-bounds]{proof} of
Lemma~\ref{lemma:L-and-D-bounds} for more details on where these
polynomials come from.

\section{Numerical experiments}\label{sec:numerical}

\begin{figure}
  \includegraphics[width=0.9375\textwidth]{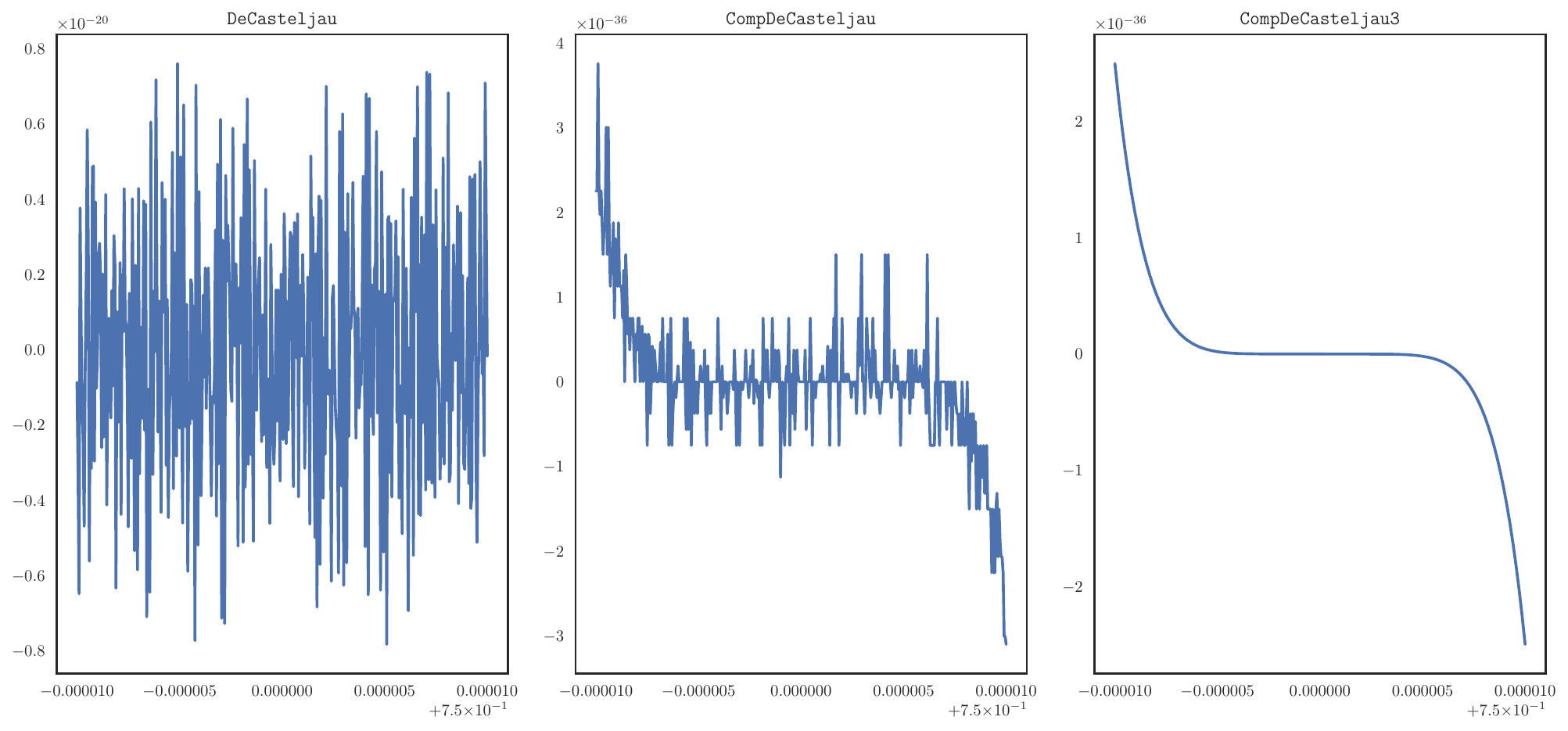}
  \centering
  \captionsetup{width=.75\linewidth}
  \caption{Evaluation of \(p(s) = (s - 1)\left(s - 3/4\right)^7\)
    in the neighborhood of its multiple root \(3/4\).}
  \label{fig:smooth-drawing}
\end{figure}

All experiments were performed in IEEE-754 double precision.
As in \cite{Jiang2010}, we consider the evaluation in the neighborhood
of the multiple root of \(p(s) = (s - 1)\left(s - 3/4\right)^7\),
written in Bernstein form.
Figure~\ref{fig:smooth-drawing} shows the evaluation of \(p(s)\) at
the 401 equally spaced\footnote{It's worth noting that \(0.1\) cannot
be represented exactly in IEEE-754 double precision (or any binary
arithmetic for that matter). Hence (most of) the points of the form
\(a + b \cdot 10^{-c}\) can only be approximately represented.} points
\(\left\{\frac{3}{4} + j \frac{10^{-7}}{2}\right\}_{j=-200}^{200}\)
with \texttt{DeCasteljau} (Algorithm~\ref{alg:de-casteljau}),
\texttt{CompDeCasteljau} (Algorithm~\ref{alg:comp-de-casteljau})
and \texttt{CompDeCasteljau3} (Algorithm~\ref{alg:k-comp-de-casteljau}
with \(K = 3\)). We see that \texttt{DeCasteljau} fails to get the
magnitude correct, \texttt{CompDeCasteljau} has the right shape but
lots of noise and \texttt{CompDeCasteljau3} is able to smoothly evaluate
the function. This is in contrast to a similar figure in \cite{Jiang2010},
where the plot was smooth for the 400 equally spaced points
\(\left\{\frac{3}{4} + \frac{10^{-4}}{2} \frac{2j - 399}{399}
\right\}_{j=0}^{399}\). The primary difference is that as the interval
shrinks by a factor of \(\approx \frac{10^{-4}}{10^{-7}} = 10^3\), the
condition number goes up by \(\approx 10^{21}\) and \texttt{CompDeCasteljau}
is no longer accurate.

\begin{figure}
  \includegraphics[width=0.8125\textwidth]{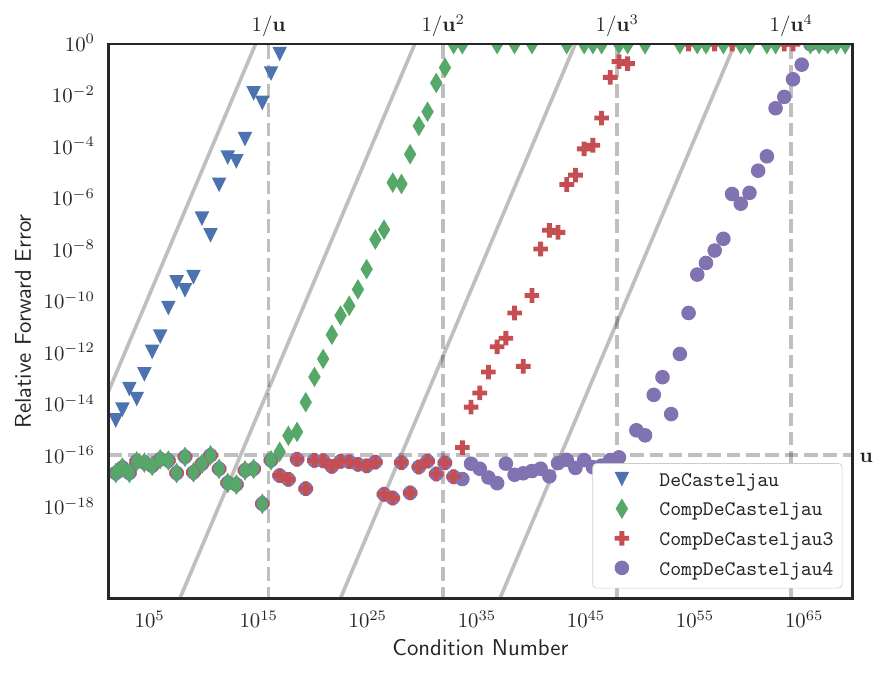}
  \centering
  \captionsetup{width=.75\linewidth}
  \caption{Accuracy of evaluation of \(p(s) = (s - 1)\left(s - 3/4\right)^7\)
    represented in Bernstein form.}
  \label{fig:compensated-k}
\end{figure}

Figure~\ref{fig:compensated-k} shows the relative forward errors compared
against the condition number. To compute relative errors, each input and
coefficient is converted to a fraction (i.e. infinite precision) and
\(p(s)\) is computed exactly as a fraction, then
compared to the corresponding computed values. Similar tools are used to
\textbf{exactly} compute the condition number, though here we can rely
on the fact that \(\widetilde{p}(s) = (s - 1)
\left(s/2 - 3/4\right)^7\). Once the relative errors and
condition numbers are computed as fractions, they are rounded to the
nearest IEEE-754 double precision value. As in \cite{Jiang2010}, we use
values \(\left\{\frac{3}{4} - (1.3)^j\right\}_{j=-5}^{-90}\)\footnote{As with
\(0.1\), it's worth noting that \((1.3)^j\) can't be represented exactly in
IEEE-754 double precision. However, this geometric series still serves a
useful purpose since it continues to raise \(\cond{p, s}\) as \(j\) decreases
away from \(0\) and because it results in ``random'' changes in the bits of
\(0.75\) that are impacted by subtracting \((1.3)^j\).}. The curves for
\texttt{DeCasteljau} and \texttt{CompDeCasteljau} trace the same paths seen
in \cite{Jiang2010}. In particular, \texttt{CompDeCasteljau} has a relative
error that is \(\bigO{\mach}\) until \(\cond{p, s}\) reaches
\(1 / \mach\), at which point the relative error increases linearly with
the condition number until it becomes \(\bigO{1}\) when
\(\cond{p, s}\) reaches \(1 / \mach^2\).
Similarly, the relative error in \texttt{CompDeCasteljau3}
(Algorithm~\ref{alg:k-comp-de-casteljau} with \(K = 3\))
is \(\bigO{\mach}\) until \(\cond{p, s}\) reaches
\(1 / \mach^2\) at which point the relative error increases linearly
to \(\bigO{1}\) when \(\cond{p, s}\) reaches \(1 / \mach^3\)
and the relative error in \texttt{CompDeCasteljau4}
(Algorithm~\ref{alg:k-comp-de-casteljau} with \(K = 4\))
is \(\bigO{\mach}\) until \(\cond{p, s}\) reaches
\(1 / \mach^3\) at which point the relative error increases linearly
to \(\bigO{1}\) when \(\cond{p, s}\) reaches \(1 / \mach^4\).

\section{Future Work}

The family of algorithms described in this paper have been implemented in
C, C++ and Python by the author (\cite{KCompensatedGitHub}). A more complete
compensated algorithms library (similar to \cite{Barrio2018}) could be quite
useful. For example, such a library could include the algorithms in the
existing literature such as the \(K\)-compensated algorithm for Horner's method
from \cite{Graillat2009}.

\section{Acknowledgements}

The author would like to thank \cite{Ogita2005}, \cite{Graillat2009} and
\cite{Jiang2010} for their papers that motivated this work. In particular,
the work of \cite{Ogita2005} reignited the path to compensated algorithms set
forth in \cite{Babuska1968}, \cite{Knuth1969} and \cite{Dekker1971}.

\section*{\refname}
\bibliography{paper}

\appendix

\section{Algorithms}\label{sec:appendix-algo}

Find here concrete implementation details on the EFTs described
in Theorem~\ref{thm:eft}. They do not use branches, nor access to the
mantissa that can be time-consuming.

\begin{breakablealgorithm}
  \caption{\textit{EFT of the sum of two floating point numbers.}}

  \begin{algorithmic}
    \Function{\(\left[S, \sigma\right] = \mathtt{TwoSum}\)}{$a, b$}
      \State \(S = a \oplus b\)
      \State \(z = S \ominus a\)
      \State \(\sigma = (a \ominus (S \ominus z)) \oplus (b \ominus z)\)
    \EndFunction
  \end{algorithmic}
\end{breakablealgorithm}

\noindent In order to avoid branching to check which among
\(\left|a\right|, \left|b\right|\) is largest, \texttt{TwoSum} uses 6 flops
rather than 3.

\begin{breakablealgorithm}
  \caption{\textit{Splitting of a floating point number into two parts.}}

  \begin{algorithmic}
    \Function{\(\left[h, \ell\right] = \mathtt{Split}\)}{$a$}
      \State \(z = a \otimes (2^r + 1)\)
      \State \(h = z \ominus (z \ominus a)\)
      \State \(\ell = a \ominus h\)
    \EndFunction
  \end{algorithmic}
\end{breakablealgorithm}

\noindent For IEEE-754 double precision floating point number, \(r = 27\)
so \(2^r + 1\) will be known before \texttt{Split} is called. In all,
\texttt{Split} uses 4 flops.

\begin{breakablealgorithm}
  \caption{\textit{EFT of the product of two floating point numbers.}}

  \begin{algorithmic}
    \Function{\(\left[P, \pi\right] = \mathtt{TwoProd}\)}{$a, b$}
      \State \(P = a \otimes b\)
      \State \(\left[a_h, a_{\ell}\right] = \mathtt{Split}(a)\)
      \State \(\left[b_h, b_{\ell}\right] = \mathtt{Split}(b)\)
      \State \(\pi = a_{\ell} \otimes b_{\ell} \ominus (((P \ominus
          a_h \otimes b_h)
          \ominus a_{\ell} \otimes b_h) \ominus a_h \otimes b_{\ell})\)
    \EndFunction
  \end{algorithmic}
\end{breakablealgorithm}

\noindent This implementation of \texttt{TwoProd} requires 17 flops.
For processors that provide a fused-multipy-add operator (\texttt{FMA}),
\texttt{TwoProd} can be rewritten to use only 2 flops:

\begin{breakablealgorithm}
  \caption{\textit{EFT of the sum of two floating point numbers with a FMA.}}
  \label{alg:two-prod-fma}

  \begin{algorithmic}
    \Function{\(\left[P, \pi\right] = \mathtt{TwoProdFMA}\)}{$a, b$}
      \State \(P = a \otimes b\)
      \State \(\pi = \mathtt{FMA}(a, b, -P)\)
    \EndFunction
  \end{algorithmic}
\end{breakablealgorithm}

\noindent The following algorithms from \cite{Ogita2005} can be used as a
compensated method for computing a sum of numbers. The first is a vector
transformation that is used as a helper:

\begin{breakablealgorithm}
  \caption{\textit{Error-free vector transformation for summation.}}
  \label{alg:vec-sum}

  \begin{algorithmic}
    \Function{\(\mathtt{VecSum}\)}{$p$}
      \State \(n = \texttt{length}(p)\)
      \For{\(j = 2, \ldots, n\)}
        \State \(\left[p_j, p_{j - 1}\right] = \mathtt{TwoSum}\left(
            p_j, p_{j - 1}\right)\)
      \EndFor
    \EndFunction
  \end{algorithmic}
\end{breakablealgorithm}

\noindent The second (\texttt{SumK}) computes a sum with results that are as
accurate as if computed in \(K\) times the working precision. It requires
\((6K - 5)(n - 1)\) floating point operations.

\begin{breakablealgorithm}
  \caption{\textit{Summation as in K-fold precision by \((K - 1)\)-fold error-free vector transformation.}}
  \label{alg:sum-k}

  \begin{algorithmic}
    \Function{\(\mathtt{result} = \mathtt{SumK}\)}{$p, K$}
      \For{\(j = 1, \ldots, K - 1\)}
        \State \(p = \mathtt{VecSum}(p)\)
      \EndFor
      \State \(\mathtt{result} = p_1 \oplus p_2 \oplus \cdots \oplus p_n\)
    \EndFunction
  \end{algorithmic}
\end{breakablealgorithm}

\noindent Since the final error \(\cdb{K - 1}\) will not track the errors
during computation, we have a non-EFT version of
Algorithm~\ref{alg:local-error-eft}:

\begin{breakablealgorithm}
  \caption{\textit{Compute the local error (non-EFT).}}
  \label{alg:local-error}

  \begin{algorithmic}
    \Function{\(\widehat{\ell} =
        \mathtt{LocalError}\)}{$e, \rho, \delta b$}
      \State \(L = \texttt{length}(e)\)
      \\
      \State \(\widehat{\ell} = e_1 \oplus e_2\)
      \For{\(j = 3, \ldots, L\)}
        \State \(\widehat{\ell} = \widehat{\ell} \oplus e_j\)
      \EndFor
      \\
      \State \(\widehat{\ell} = \widehat{\ell} \oplus \left(
          \rho \otimes \delta b\right)\)
    \EndFunction
  \end{algorithmic}
\end{breakablealgorithm}

\section{Proof Details}\label{sec:appendix-proof-details}

\begin{proof}[Proof of Lemma~\ref{lemma:ell-tilde}]
We'll start with the \(F = 1\) case. Recall where the terms originate:
\begin{align}
\left[P_1, e_1\right] &= \mathtt{TwoProd}\left(\widehat{r},
  \widehat{b}_j^{(k + 1)}\right) \\
\left[P_2, e_2\right] &= \mathtt{TwoProd}\left(s,
  \widehat{b}_{j + 1}^{(k + 1)}\right) \\
\left[\widehat{b}_j^{(k)}, e_3\right] &= \mathtt{TwoSum}\left(P_1, P_2\right).
\end{align}
Hence Theorem~\ref{thm:eft} tells us that
\begin{align}
\left|P_1\right| &\leq (1 + \mach)\left|\widehat{r} \cdot
  \widehat{b}_j^{(k + 1)}\right| \leq (1 + \mach)^2 (1 - s)
  \left|\widehat{b}_j^{(k + 1)}\right| \\
\left|e_1\right| &\leq \mach \left|\widehat{r} \cdot
  \widehat{b}_j^{(k + 1)}\right| \leq \mach(1 + \mach)(1 - s) \left|
  \widehat{b}_j^{(k + 1)}\right| \\
\left|P_2\right| &\leq (1 + \mach) s \left|\widehat{b}_{j + 1}^{(k + 1)}\right| \\
\left|e_2\right| &\leq \mach s \left|\widehat{b}_{j + 1}^{(k + 1)}\right| \\
\left|e_3\right| &\leq \mach \left|P_1\right| + \mach\left|P_2\right| \\
\left|\rho \cdot \widehat{b}_j^{(k + 1)}\right| &\leq
(1 + \mach)(1 - s) \left|\widehat{b}_j^{(k + 1)}\right|.
\end{align}
In general, we can swap \(\mach\left|P_j\right|\) for
\((1 + \mach)\left|e_j\right|\) based on how closely related the bound
on the result and the bound on the error are. Thus
\begin{align}
\widetilde{\ell}_{1, j}^{(k)} &= \left|e_1\right| + \left|e_2\right| +
  \left|e_3\right| + \left|\rho \cdot \widehat{b}_j^{(k + 1)}\right| \\
&\leq (2 + \mach)\left(\left|e_1\right| + \left|e_2\right|\right) +
  (1 + \mach)(1 - s) \left|\widehat{b}_j^{(k + 1)}\right| \\
&\leq \left[(1 + \mach)^3 - 1\right] (1 - s) \left|
  \widehat{b}_j^{(k + 1)}\right| + \left[(1 + \mach)^2 - 1\right] s \left|
  \widehat{b}_{j + 1}^{(k + 1)}\right| \\
&\leq \gamma_3 \left((1 - s) \left|\widehat{b}_j^{(k + 1)}\right| +
  s \left|\widehat{b}_{j + 1}^{(k + 1)}\right|\right).
\end{align}
For \(\widetilde{\ell}_{F + 1}\), we want to relate the ``current'' errors
\(e_1, \ldots, e_{5F + 3}\) to the ``previous'' errors \(e_1',
\ldots, e_{5F - 2}'\) that show up in \(\widetilde{\ell}_F\). In the same
fashion as above, we track where the current errors come from:
\begin{align}
\left[S_1, e_1\right] &= \mathtt{TwoSum}\left(e_1', e_2'\right) \\
\left[S_2, e_2\right] &= \mathtt{TwoSum}\left(S_1, e_3'\right) \\
&\mathrel{\makebox[\widthof{=}]{\vdots}} \nonumber \\
\left[S_{5F - 3}, e_{5F - 3}\right] &=
  \mathtt{TwoSum}\left(S_{5F - 4}, e_{5F - 2}'\right) \\
\left[P_{5F - 2}, e_{5F - 2}\right] &= \mathtt{TwoProd}\left(\rho,
  \cdb{F - 1}_j^{(k + 1)}\right) \\
\left[\widehat{\ell}_{F, j}^{(k)}, e_{5F - 1}\right] &=
  \mathtt{TwoSum}\left(S_{5F - 3}, P_{5F - 2}\right) \\
\left[P_{5F}, e_{5F}\right] &= \mathtt{TwoProd}\left(s,
  \cdb{F}_{j + 1}^{(k + 1)}\right) \\
\left[S_{5F + 1}, e_{5F + 1}\right] &=
  \mathtt{TwoSum}\left(\widehat{\ell}_{F, j}^{(k)}, P_{5F}\right) \\
\left[P_{5F + 2}, e_{5F + 2}\right] &= \mathtt{TwoProd}\left(\rho,
  \cdb{F}_j^{(k + 1)}\right) \\
\left[\cdb{F}_j^{(k)}, e_{5F + 3}\right] &= \mathtt{TwoSum}\left(
  S_{5F + 1}, P_{5F + 2}\right).
\end{align}
Arguing as we did above, we start with
\(\left|e_1\right| \leq \mach \left|e_1'\right| + \mach \left|e_2'\right|\)
and build each bound recursively based on the previous, e.g.
\(\left|e_2\right| \leq \mach \left|S_1\right| + \mach \left|e_3'\right| \leq
(1 + \mach) \mach \left|e_1'\right| + (1 + \mach) \mach \left|e_2'\right| +
\mach \left|e_3'\right|\). Proceeding in this fashion, we find
\begin{align}
\widetilde{\ell}_{F + 1, j}^{(k)} &= \left|e_1\right| + \cdots +
  \left|e_{5F + 3}\right| + \left|\rho \cdot \cdb{F}_j^{(k + 1)}\right| \\
&\leq \gamma_{5F} \left|e_1'\right| + \gamma_{5F} \left|e_2'\right| +
  \gamma_{5F - 1} \left|e_3'\right| + \cdots +
  \gamma_4 \left|e_{5F - 2}'\right| +
  \gamma_4 \left|\rho \cdot \cdb{F - 1}_j^{(k + 1)}\right| \\
&\qquad + \gamma_3 (1 - s) \left|
  \cdb{F}_j^{(k + 1)}\right| + \gamma_3 s \left|
  \cdb{F}_{j + 1}^{(k + 1)}\right| \\
&\leq \gamma_3 \left(
  (1 - s) \left|\cdb{F}_j^{(k + 1)}\right| +
  s \left|\cdb{F}_{j + 1}^{(k + 1)}\right|\right) +
  \gamma_{5F} \cdot \widetilde{\ell}_{F, j}^{(k)}
\end{align}
as desired.
\end{proof}

\begin{proof}[Proof of Lemma~\ref{lemma:L-and-D-bounds}]\label{proof:L-and-D-bounds}
First, note that for \textbf{any} sequence \(v_0, \ldots, v_{k + 1}\) we
must have
\begin{equation}
\sum_{j = 0}^k \left[(1 - s) v_j + s v_{j + 1}\right] B_{j, k}(s) =
\sum_{j = 0}^{k + 1} v_j B_{j, k + 1}(s).
\end{equation}
For example of this in use, via \eqref{ell-tilde-1}, we have
\begin{equation}
  L_{1, k} \leq \gamma_3 \sum_{j = 0}^{k + 1} \left|
  \widehat{b}_j^{(k + 1)}\right| B_{j, k + 1}(s).
\end{equation}
In order to work with sums of this form, we define Bernstein-type
sums related to \(L_{F, k}\):
\begin{align}
D_{0, k} &\coloneqq \sum_{j = 0}^k \left|\widehat{b}_j^{(k)}\right|
B_{j, k}(s) \\
D_{F, k} &\coloneqq \sum_{j = 0}^k \left|\cdb{F}_j^{(k)}\right| B_{j, k}(s).
\end{align}
Hence Lemma~\ref{lemma:ell-tilde} gives
\begin{align}
L_{1, k} &\leq \gamma_3 D_{0, k + 1} \label{ell-1-k} \\
L_{F + 1, k} &\leq \gamma_3 D_{F, k + 1} + \gamma_{5F} L_{F, k}
\label{ell-F-k}
\end{align}
In addition, for \(F \geq 1\) since
\begin{align}
\cdb{F}_j^{(k)} &= \widehat{\ell}_{F, j}^{(k)} \oplus \left(
  s \otimes \cdb{F}_{j + 1}^{(k + 1)}\right) \oplus \left((1 \ominus s) \otimes
  \cdb{F}_{j}^{(k + 1)}\right) \\
&= (1 - s) \cdot \cdb{F}_j^{(k + 1)}(1 + \theta_3) +
  s \cdot \cdb{F}_{j + 1}^{(k + 1)}(1 + \theta_3) +
  \widehat{\ell}_{F, j}^{(k)} (1 + \theta_2)
\end{align}
we have
\begin{equation}\label{df-first}
D_{F, k} \leq (1 + \gamma_3) D_{F, k + 1} + (1 + \gamma_2) \sum_{j = 0}^k
\left|\widehat{\ell}_{F, j}^{(k)}\right| B_{j, k}(s).
\end{equation}
Since \(\ell_{F, j}^{(k)}\) has \(5F - 1\) terms (only the last of which
involves a product), the terms in the computed value will be involved in
at most \(5F - 2\) flops, hence
\(\left|\widehat{\ell}_{F, j}^{(k)}\right| \leq
\left(1 + \gamma_{5F - 2}\right) \widetilde{\ell}_{F, j}^{(k)}.\)
Combined with \eqref{df-first} and the fact that there is no local error
when \(F = 0\), this means
\begin{align}
D_{0, k} &\leq (1 + \gamma_3) D_{0, k + 1} \label{d-0-k} \\
D_{F, k} &\leq (1 + \gamma_3) D_{F, k + 1} + (1 + \gamma_{5F}) L_{F, k}.
\label{d-F-k}
\end{align}
The four inequalities \eqref{ell-1-k}, \eqref{ell-F-k}, \eqref{d-0-k}
and \eqref{d-F-k} allow us to write all bounds in terms of
\(D_{0, n} = \widetilde{p}(s)\) and \(D_{F, n} = 0\). From \eqref{d-0-k}
we can conclude that \(D_{0, n - k} \leq \left(1 + \gamma_{3k}\right) \cdot
\widetilde{p}(s)\) and from \eqref{ell-1-k} that \(L_{1, n - k} \leq
\gamma_3 \left(1 + \gamma_{3(k - 1)}\right) \cdot \widetilde{p}(s)\).

To show the bounds for higher values of \(F\), we'll assume we have
bounds of the form
\(D_{F, n - k} \leq \left(q_F(k) \mach^F + \bigO{\mach^{F + 1}}\right) \cdot
\widetilde{p}(s)\) and
\(L_{F, n - k} \leq \left(r_F(k) \mach^F + \bigO{\mach^{F + 1}}\right) \cdot
\widetilde{p}(s)\) for two families of polynomials \(q_F(k), r_F(k)\). We
have \(q_0(k) = 1\) and \(r_1(k) = 3\) as our base cases and can build from
there. To satisfy \eqref{d-F-k}, we'd like
\(q_F(k) = q_F(k - 1) + r_F(k)\)
and for \eqref{ell-F-k}
\(r_{F + 1}(k) = 3 q_F(k - 1) + 5 F r_F(k)\).
Since the forward difference \(\Delta q_F(k) = r_F(k + 1)\) is known,
we can inductively solve for \(q_F\) in terms of \(q_F(0)\). But
\(D_{F, n} = 0\) gives \(q_F(0) = 0\).

For example, since we have \(r_1(k) = 3 \binom{k}{0}\) we'll have
\(q_1(k) = 3 \binom{k}{1}\). Once this is known
\begin{equation}
r_2(k) = 3 q_1(k - 1) + 5 r_1(k) = 3 \cdot 3 \binom{k - 1}{1} +
5 \cdot 3 \binom{k}{0} = 9 \binom{k}{1} + 6 \binom{k}{0}.
\end{equation}
If we write these polynomials in the ``falling factorial'' basis of
forward differences, then we can show that
\begin{equation}
r_F(k) = 3^F \binom{k}{F} + \cdots
\end{equation}
which will complete the proof of the first inequality. To see this, first
note that for a polynomial in this basis
\(f(k) = A \binom{k}{d} + B \binom{k}{d - 1} + C \binom{k}{d - 2} +
D \binom{k}{d - 3} + \cdots\) we have
\begin{align}
f(k + 1) &= A \binom{k}{d} + (A + B) \binom{k}{d - 1} +
  (B + C) \binom{k}{d - 2} + (C + D) \binom{k}{d - 3} + \cdots \\
f(k - 1) &= A \binom{k}{d} + (B - A) \binom{k}{d - 1} +
  (C - B + A) \binom{k}{d - 2} + (D - C + B - A) \binom{k}{d - 3} + \cdots
\end{align}
Using these, we can show that if
\(r_F(k) = \sum_{j = 0}^{F - 1} c_j \binom{k}{j}\) then
\begin{align}
q_F(k) &= c_{F - 1} \binom{k}{F} + \sum_{j = 1}^{F - 1}
  (c_j + c_{j - 1}) \binom{k}{j} \\
r_{F + 1}(k) &= 3 \left[-c_0 \binom{k}{0} +
  \sum_{j = 1}^F c_{j - 1} \binom{k}{j}\right] +
5F \left[\sum_{j = 0}^{F - 1} c_j \binom{k}{j}\right] =
3 c_{F - 1} \binom{k}{F} + \cdots
\end{align}
Under the inductive hypothesis \(c_{F - 1} = 3^F\) so that
the lead term in \(r_{F + 1}(k)\) is \(3 c_{F - 1} \binom{k}{F}
= 3^{F + 1} \binom{k}{F}\).

For the second inequality, we'll show that
\begin{equation}
\sum_{k = 0}^{n - 1} \gamma_{3k + 5F} L_{F, k} \leq
  \left[q_{F + 1}(n) \mach^{F + 1} +
  \bigO{\mach^{F + 2}}\right] \cdot \widetilde{p}(s)
\end{equation}
and then we'll have our result since we showed above that
\(q_{F + 1}(n) = 3^{F + 1} \binom{n}{F + 1} + \bigO{n^F}\). Since
\(\gamma_{3k + 5F} L_{F, k} \leq (3k + 5F) L_{F, k} \mach +
\bigO{\mach^{F + 2}} \widetilde{p}(s)\) it's enough to consider
\begin{equation}
\sum_{k = 0}^{n - 1} (3k + 5F) r_F(n - k) =
\sum_{k = 1}^n (3(n - k) + 5F) r_F(k).
\end{equation}
Since \(q_F(k) = q_F(k - 1) + r_F(k)\) and \(q_F(0) = 0\) we have
\(q_{F}(n) = \sum_{k = 1}^n r_{F}(k)\) thus
\begin{equation}
q_{F + 1}(n) = \sum_{k = 1}^n r_{F + 1}(k)
= \sum_{k = 1}^n 3 q_F(k - 1) + 5 F r_F(k)
= \sum_{k = 1}^n 3 \left[\sum_{j = 1}^{k - 1} r_F(j)\right] + 5 F r_F(k).
\end{equation}
Swapping the order of summation and grouping like terms, we have our
result.
\end{proof}

\end{document}